\documentclass[12pt]{article}

\usepackage{authblk}
\usepackage{amssymb}
\usepackage{amsmath}
\usepackage{amsthm}
\usepackage{color}
\usepackage{graphicx}
\usepackage{enumerate}
\usepackage{subfigure}
\usepackage{url}

\usepackage[utf8]{inputenc}

\usepackage[pdfstartview=FitB,
                    pdfpagemode=FitB, pagebackref]{hyperref} 
                    
\newtheorem{thm}                   {Theorem} 
\newtheorem{prop}         [thm]{Proposition}
\newtheorem{lem}          [thm]{Lemma}

\theoremstyle{definition} 
\theoremstyle{definition}    
\theoremstyle{definition}    
\theoremstyle{definition}

\title{The Grünbaum--Rigby configuration as a special Kárteszi configuration}

\author[1]{G\'abor G\'evay}
\author[2,3]{György Kiss}
\author[3,4,5]{Toma\v{z} Pisanski}
\affil[1]{University of Szeged, Szeged, Hungary}
\affil[2]{Eötvös Loránd University, Budapest, Hungary} 
\affil[3]{FAMNIT, University of Primorska, Koper, Slovenia}
\affil[4]{IAM, University of Primorska, Koper, Slovenia}
\affil[5]{Institute of Mathematics, Physics and Mechanics, Ljubljana, Slovenia}

\begin{document}

\maketitle

\begin{abstract}

In 1990, Branko Grünbaum and John Rigby presented a 4-configuration, known today as the 
\emph{Grünbaum--Rigby configuration}; it is denoted by $\mathrm{GR}(21_4)$. Independently and earlier, in 
1986, Ferenc Kárteszi published a paper in which he proved a theorem in real geometry that gives rise to a 
series of 4-configurations $\mathrm{K}(n;\ell,m)$. In an even earlier paper from 1964, he presented a figure 
which is essentially the same as that given by Grünbaum and Rigby. In this paper, we explore some properties 
of the \emph{K\'arteszi configurations} and in particular show that $\mathrm{GR}(21_4)$ is isomorphic to 
$\mathrm{K}(7;2,3)$.
We present a theorem that gives necessary and sufficient conditions on parameters $n,\ell,m$ such that the 
corresponding configuration $\mathrm{K}(n;\ell,m)$ is realisable as a geometric polycyclic configuration with 
$n$-fold rotational symmetry and no extra incidences.

\end{abstract}
\medskip



\section{Introduction} \label{sect:intro}


Ferenc Kárteszi (1907--1989) is considered by many as the father of the Hungarian school of finite geometry. 
In 2008 a special issue of \emph{Contributions to Discrete Mathematics} commemorating the centenary of the 
birth of Ferenc Kárteszi was edited by Gábor Korchmáros and Tamás Szőnyi \cite{KSz2008} in which a short 
biography of Kárteszi appeared~\cite{KSz2008a}. In 1972 Kárteszi published a book on finite geometries in 
which one chapter is dedicated to configurations \cite{Kar1972}, first in Hungarian that was later translated 
into English (1976), Italian (1978), and Russian (1980).

In this note we would like to emphasize an important and unfortunately overlooked contribution of Kárteszi 
to the theory of point-line configurations.  

The first book on configurations was written by F.\ Levi~\cite{Levi} in 1929 in German. Later, in 1932 
Hilbert and Cohn-Vossen~\cite{HCV1932} dedicated a whole chapter to projective configurations. The next 
important step is the paper by Coxeter~\cite{Cox1950} in which the notions of \emph{Levi graph} and 
\emph{Menger graph} were introduced (for these graphs, see also~\cite{PS}).

The book of Branko Grünbaum~\cite{Gru} represents the next milestone of the theory of point-line configurations, 
which makes a synthesis and re-organization of contemporary work in the field that is based on the author's 
work on configurations. In some cases, it departs from the established terminology. For instance, the term 
\emph{symmetric configuration} is replaced by \emph{balanced configuration}. The distinction between 
combinatorial, topological, and geometric configuration is introduced. The newly introduced term 
\emph{$k$-configuration} also represents the configurations having $k$ points on a line and $k$ 
lines through a point. 

One chapter is dedicated to the study of $3$-configurations, while another is dedicated to $4$-configurations.

The existence problem of $3$-configurations was solved a long time ago. However, the existence problem 
originated by Grünbaum in his paper with Rigby~\cite{GR} in which they provided the existence of a 
geometric point-line $(21_4)$ configuration, a combinatorial version of which was studied earlier by 
Felix Klein. Due to the paper cited, this configuration is called the \emph{Gr\"unbaum--Rigby configuration}; 
we denote it by $\mathrm{GR}(21_4)$ (for a drawing of it, see Figure~\ref{fig:GR}). 
\begin{figure}[!h]
\begin{center}
\includegraphics[width=0.66\textwidth]{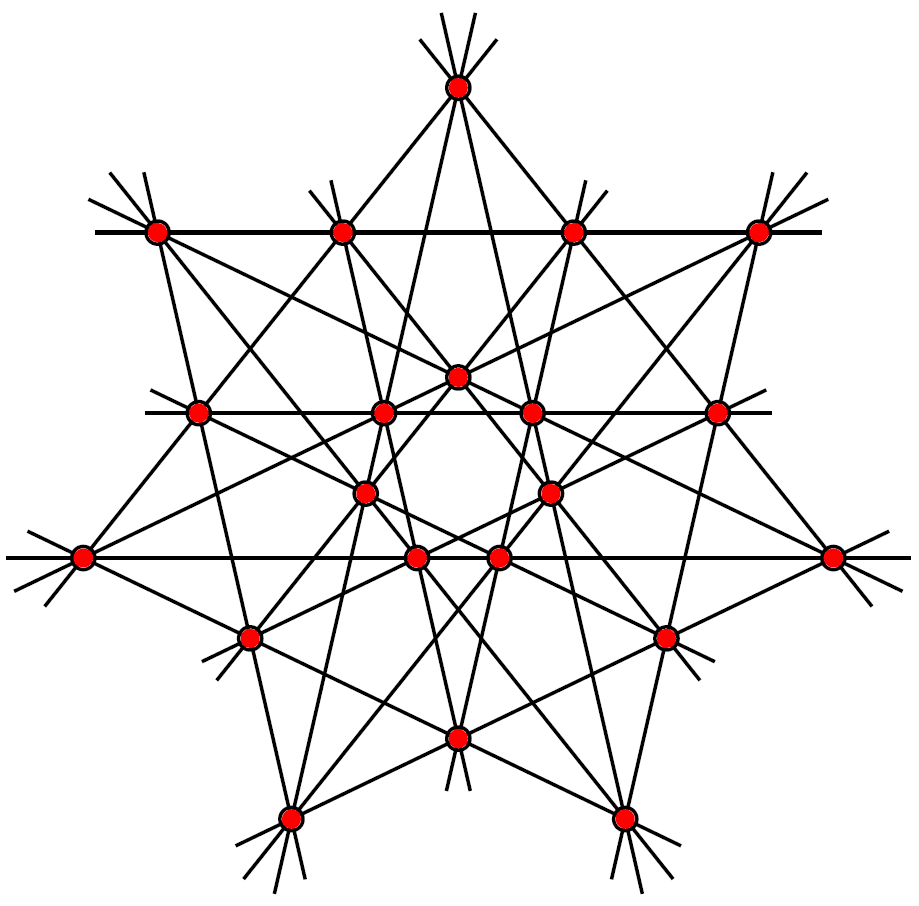}
\caption{The Gr\"unbaum--Rigby configuration $\mathrm{GR}(21_4)$.}
 \label{fig:GR}
\end{center}
\end{figure}

Later, he and some of his coworkers and students wrote a series of papers, mostly in \emph{Geombinatorics}, 
on the problem of values of $n$ for which at least one $(n_4)$ configuration exists. 

In 2003, the same problem was independently attacked by a general theory of \emph{polycyclic configurations}, 
introduced by Boben and Pisanski~\cite{BoPi2003}. A combinatorial configuration is \emph{polycyclic} if it 
has a semi-regular automorphism, i.e.\ an automorphism of order $m$ such that each orbit consisting of either 
points or lines has the same order $m$. A geometric configuration is polycyclic if its underlying combinatorial 
configuration is polycyclic and the corresponding semi-regular automorphism can be realised as a $m$-fold rotation. 
In Grünbaum's book \cite{Gru} a special class of polycyclic configurations is named 
\emph{$d$-astral configurations}~\cite[Definition 3.5.1]{Gru}. 
His student Leah Berman named them $d$-celestial configurations. 
Two notations were used: Boben and Pisanski specified them as 
$$C_4(m;(p_1,p_2, \ldots p_d),(q_1,q_2, \ldots, q_d),t)$$ 
configurations, whereas Grünbaum and Berman denoted them as 
$$m\# (p_1,q_1;p_2,q_2;\ldots; p_d,q_d)$$. 
Since geometric $d$-celestial configurations exist only if parameter 
$$t = ((p_1+p_2+ \ldots +p_d)-(q_1+q_2+ \ldots + q_d))/2$$ is integer,
the parameter $t$ can be dropped and the latter notation prevailed.

Let 
$$m\#(p_1,q_1;p_2,q_2;\ldots; p_d,q_d)$$ 
be a $d$-celestial $4$-configuration, and let 
$$P = (p_1,p_2, \ldots,p_d)$$ 
and 
$$Q = (q_1,q_2,\ldots,q_d)$$ 
be the multisets defined by the configuration parameters.
The collection of parameters producing a genuine geometric celestial configuration (with no extra incidences) is 
called a \emph{valid set of parameters}.

A couple of years ago, Berman and Berardinelli~\cite{BB} gave a set of four axioms that identify valid parameters 
for celestial 4-configurations. They are conveniently named: even condition (A1), order condition (A2), cosine 
condition (A3), and substring condition (A4). These axioms already appear in a less systematic way in~\cite{BoPi2003} 
by Boben and Pisanski. If we are interested in connected configurations, anothr axiom 
reads $\gcd(P,Q,m) = 1$ and may be denoted as connectivity condition (A5), has to be added. 

Grünbaum defines the configuration to be \emph{trivial} celestial configuration if the multisets $P$ 
and $Q$ are equal. For trivial celestial configurations, the axioms (A1) and (A3) are always satisfied. 
The necessary conditions for the parameters to be valid for a trivial $3$-celestial configuration 
then $P = Q = \{a,b,c\}$ has to be a set (and not a multiset) and the configuration has to be described as $m\#
(a,b;c,a;b,c)$ to satisfy axiom (A2).  Moreover, any of the six permutations of $P$ gives rise to the same 
configuration. Hence, we may choose $a < b < c$. For larger values of $d$ the same multiset $P$ may be used 
for different configurations. It turns out that some conditions need to be imposed on the parameters to 
guarantee that the structure is indeed a combinatorial $4$-configuration. However, when realizing them 
geometrically, some false incidences may appear.

Quite recently the second author pointed to us the paper~\cite{Kar1986} written in Italian with a theorem of 
Kárteszi that implies the existence of the Grünbaum--Rigby configuration. A reference~\cite{Kar65} in this 
paper was difficult to obtain. It was published in 1964 in a lesser known journal from Messina. When we got 
the copy of that paper thanks to Gábor Korchmáros, we were stunned by the fact that it not only describes 
a $(21_4,21_4)$ configuration, which in our notation is the  $\mathrm{GR}(21_4)$ configuration, but also provides the 
familiar picture of the Grünbaum--Rigby configuration (see our Figure~\ref{fig:GR} above).  

Here we note that one has to be careful, since the Grünbaum--Rigby configuration has been shown to be not the 
only $(21_4)$ example; in the paper \cite{BGP2024} another non-equivalent $(21_4)$ configuration was constructed.

The Kárteszi Theorem (see next section) is a basis for the construction of a three-parameter family of 
$4$-configurations that we denote by $\mathrm{K}(n;\ell,m)$. Here we propose the name \emph{Kárteszi 
configurations} for them. This family forms a subfamily of the trivial $3$-celestial $4$-configurations; 
namely, $\mathrm{K}(n;\ell,m) = n \#(1,\ell;m,1;\ell,n)$. Any Kárteszi configuration is connected. 
Any Kárteszi configuration with $1 < a< b$ is geometrically realizable  with no extra incidences if and only 
if the following Proposition is satisfied.

\begin{prop}
The Kárteszi configuration $\mathrm{K}(n;\ell,m)$ has no extra incidence if and only if there is no 
integer $x$ such that the 2-celestial 4-configuration (= astral configuration) $n \#(\ell,m,x)$ exists.
\end{prop}


In such a case all three configurations 
$\mathrm{K}(n;\ell,m),\mathrm{K}(n;\ell,x),\mathrm{K}(n;m,x)$ are realizable.

Berman \cite{Be2001} was the first to characterize valid parameters of astral 4-configurations. Her result is 
presented in Grünbaum's book as Theorem 3.6.1. 
By selecting the configurations having one parameter equal to 1, our Theorem \ref{thm:extra-inc} is obtained.
However, we give an independent proof of  this theorem that is based on Kárteszi theorem and a result of Poonen 
and Rubinstein.
Figure \ref{fig:KartesziPaper} taken from~\cite{Kar65} shows without doubt that Kárteszi constructed the configuration 
$\mathrm{K}(7;2,3)$ in 1964/65. 
\begin{figure} [!h]
\begin{center}
\includegraphics[width=0.785\linewidth]{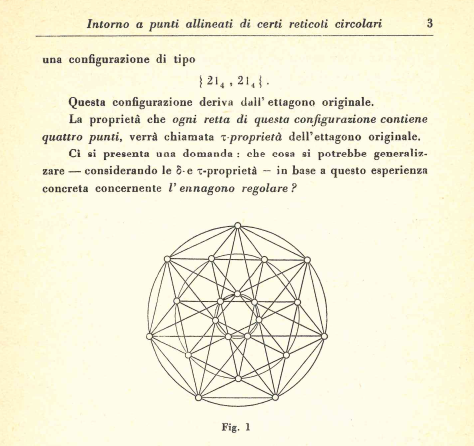}
   \caption{Drawing of $\mathrm \mathrm{K}(7;2,3)$ taken from the paper~\cite{Kar65} of Kárteszi.}
   \label{fig:KartesziPaper}
\end{center}
\end{figure}

It seems impossible to have been overlooked for such a long time. Unfortunately, there were too many adverse
factors. 

\begin{itemize}
\item 
The paper in which this configuration appears is written in Italian.
\item
The journal is not covered by \emph{zbMATH} or \emph{MathRev}.
\item
Kárteszi in his book does not mention any 4-configuration. Apparently, he did not find his result or later 
theorem sufficiently interesting. Hence, even when the book was translated into English, his discovery 
remained hidden. The term ``geometric configuration" that is used in his book nowadays represents a 
``linear configuration".
\item
The book does not cite Coxeter's seminal paper of 1950~\cite{Cox1950}. So, he uses the term \emph{incidence graph} 
and not \emph{Levi graph}. 
\item
Grünbaum in his book does not mention that of Kárteszi, although in the latter a whole chapter is devoted to 
configurations.
\item
On p.\ 212, Grünbaum mentions the 1986 paper of Kárteszi but does not give him credit for $(21_4)$.
\end{itemize}

Our intention is to point out the significance of the Kárteszi theorem to the theory of celestial configurations.


\section{From Kárteszi Theorem to Kárteszi configurations} \label{sect:K}


For the sake of completeness, we present a short proof of the K\'arteszi Theorem.
We remark that in this theorem, and throughout the paper, by a regular $n$-gon, we mean a \emph{convex} 
regular $n$-gon.

\begin{thm}[Kárteszi, 1964~\cite{Kar65}]
\label{thm:chord-belt}
Let $\mathcal{A}_1=A_0A_1\dots A_{n-1}$ be a regular $n$-gon on the classical Euclidean plane. For any
$k\in \{1,2,\dots , \lfloor\frac{n-1}{2}\rfloor \}$, 
the $k$-th diagonals of 
$\mathcal{A}_1,$ $A_0A_k,\, A_1A_{k+1},\, \dots ,A_{n-1}A_{k+1}$, 
form another regular $n$-gon $\mathcal{A}_k$, which we call the $k$-th $n$-gon of $\mathcal{A}_1$. 
Then the $m$-th diagonals of the $\ell $-th $n$-gon and the $\ell$-th diagonals of the $m$-th 
$n$-gon of $\mathcal{A}_1$ have a common line for any 
$\ell ,m \in \{1,2,\dots, \lfloor\frac{n-1}{2}\rfloor \}$.
\end{thm}

\begin{figure} [h]
    \centering
    \includegraphics[width=0.825\linewidth]{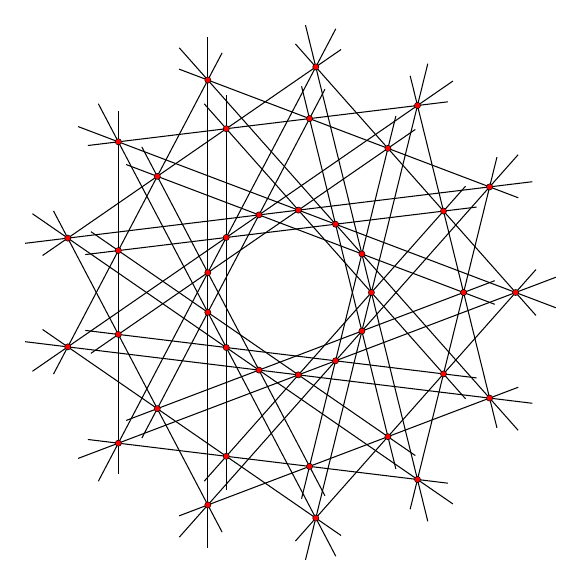}
    \caption{Example of an application of Kárteszi Theorem for $n  = 13,\, \ell = 3$ and $m = 5$ producing 
             $\mathrm K(13;3,5)$.}
    \label{fig:example(13,3,5)}
\end{figure}

First, we prove a lemma.

\begin{lem}
\label{lem:rotate}
Let 
$\mathcal{R}=R_0R_1\dots R_{n-1}$ be a regular $n$-gon whose centre is $C$, and let 
$\mathcal{T}=T_0T_1\dots T_{n-1}$ be the $k$-th $n$-gon of $\mathcal{R}$, where
$T_j=R_{j-1}R_{j-1+k}\cap R_jR_{j+k}$, and let $\beta =\frac{2\pi}{n}$. Let $\varphi _k$ denote the rotational 
similarity $\varphi _k$ with centre $C$, angle $\frac{k-1}{2}\beta$, and ratio 
$\frac{\cos \frac{k\beta }{2}}{\cos \frac{\beta }{2}}$.
Then $\varphi _k(R_j)=T_j$ and  $\varphi _k$ map the line
$R_jR_{j+1}$ to the line $R_jR_{j+k}$ for $j=0,1,\dots ,n-1$. 
If  $F$ and $G$ denote the midpoints of the line segments $R_0R_1$
and $R_0R_k$, respectively, then  $\varphi _k(F)=G.$ 
\end{lem}
\begin{figure} [h]
    \centering
    \includegraphics[width=0.825\linewidth]{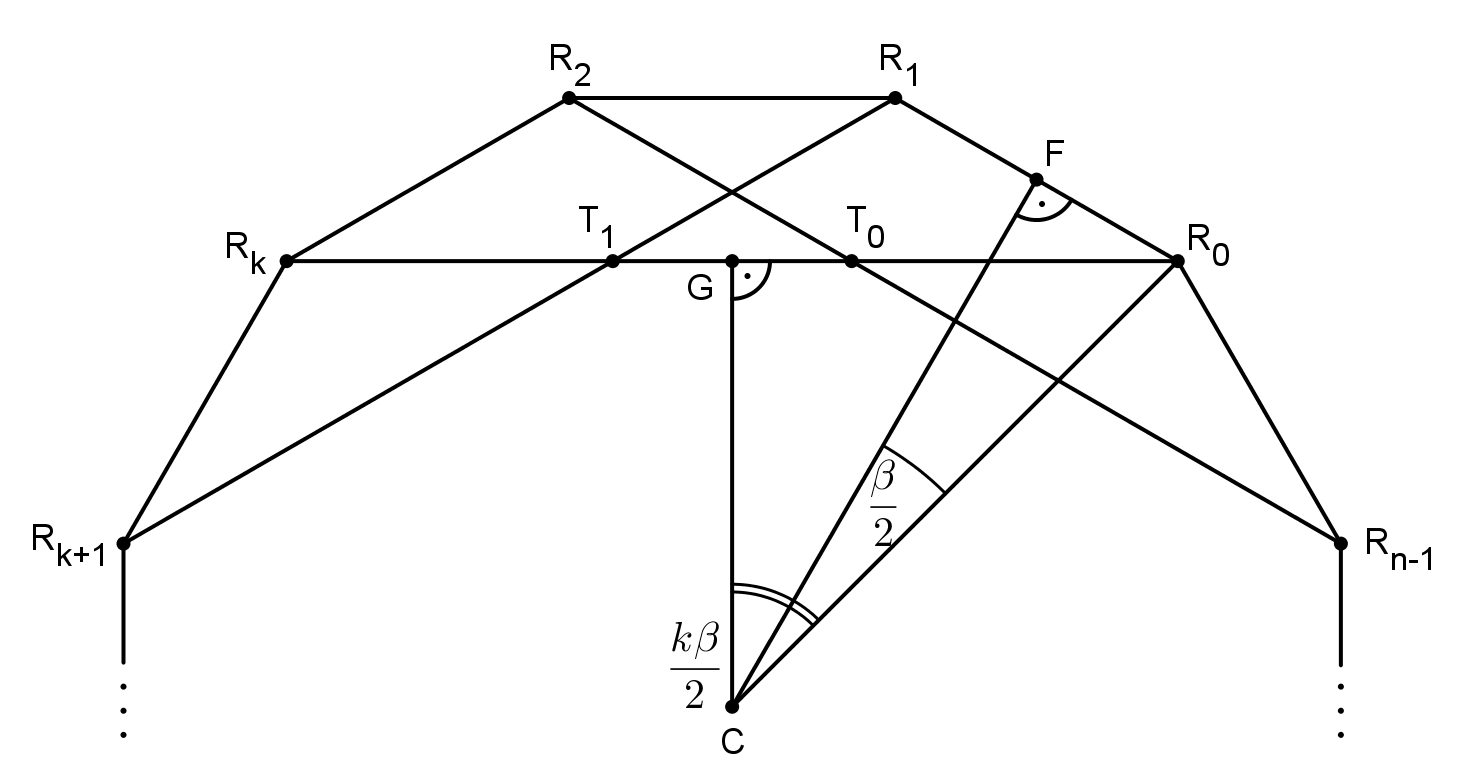}
    \caption{Illustration for Lemma~\ref{lem:rotate}.}
    \label{fig:placeholder}
\end{figure}

\begin{proof}
Due to the rotational symmetry of $\mathcal{R}$, it is sufficient to prove the statements for $j=0.$ The line $CG$ is an axis of symmetry of $\mathcal{R}$, hence 
it is also an axis of symmetry of $\mathcal{T}$. So $G$ is the midpoint of the line
segment $T_0T_1$. Since
$R_0CF\angle =\frac{1}{2}R_0CR_1\angle =\frac{\beta}{2}$ and 
$R_0CG\angle =\frac{1}{2}R_0CR_k\angle =\frac{k\beta}{2},$ we get
$$R_0CT_0\angle = FCG\angle =\frac{k-1}{2}\beta.$$ 

The triangles $R_0FC$ and $R_0GC$ are right-angled, and hence 
$CF =R_0C\cos \frac{\beta }{2}$ and $CG =R_0C\cos \frac{k\beta }{2}.$ So
$$T_0C\colon R_0C= GC \colon FC = \cos \frac{k\beta }{2}\colon \cos \frac{\beta }{2}.$$
Thus, $\varphi _k(R_0)=T_0$ and $\varphi _k(F)=G$.
The line $R_0R_{1}$ is mapped to the line $R_0R_{k},$ because the points $T_0$ and $T_1$ are on the line $R_0R_{k}$. 
The lemma is proven.
\end{proof}

\smallskip

\noindent
{\bf Proof of Theorem~\ref{thm:chord-belt}.}
Let the vertices of $\mathcal{A}_m$ be
$B_j=A_{j-1}A_{j-1+m}\cap A_jA_{j+m},$ and the vertices of $\mathcal{A}_{\ell }$ 
be $D_j=A_{j-1}A_{j-1+\ell }\cap A_jA_{j+\ell }.$ Due to the rotational symmetry, it is sufficient to prove that the 
lines $B_0B_{\ell }$ and $D_0D_m$ coincide.

Take line $A_0A_1$ and apply Lemma \ref{lem:rotate} twice. 
First, for the $n$-gon $\mathcal{A}_1$ and the transformation $\varphi _{\ell}$, then apply to the $n$-gon 
$\mathcal{A}_{\ell }$ and the transformation $\varphi _m$. 
$$\varphi _{m}(\varphi _{\ell }(A_0A_1))=\varphi _{m}(A_0A_{\ell })=
\varphi _{m}(D_0D_{1})=D_0D_m.$$ 
Second, for the $n$-gon $\mathcal{A}_1$ and the transformation $\varphi _{m}$, then apply to the $n$-gon 
$\mathcal{A}_{m}$ and the transformation $\varphi _{\ell }$. 
$$
\varphi _{\ell }(\varphi _{m}(A_0A_1))=\varphi _{\ell }(A_0A_{m})=
\varphi _{m}(B_0B_{1})=B_0B_{\ell }.
$$ 

Since the rotational 
similarities $\varphi _m$ and $\varphi _{\ell}$ have the same centre, they commute. Thus
$\varphi _{\ell }(\varphi _{m}(A_0A_1))=\varphi _{m}(\varphi _{\ell }(A_0A_1)),$ hence
the lines $B_0B_{\ell }$ and $D_0D_m$ coincide, which proves the theorem.
\hfill{$\Box $}

\smallskip

Kárteszi Theorem provides a simple way to construct geometric point-line configurations. 
Using the notation of Theorem~\ref{thm:chord-belt}, take the set of $3n$ elements consisting 
of the vertices of $\mathcal{A}_{1}, \mathcal{A}_{\ell}$ and $\mathcal{A}_{m}$, and take the 
set of  $3n$ elements consisting of the $\ell$-th and $m$-th  diagonals of $\mathcal{A}_{1}$ and 
the common lines of the $\ell$-th diagonals of $\mathcal{A}_m$ and the $m$-th  diagonals of $\mathcal{A}_{\ell }$. 
Each point is incident with at least four lines, and similarly each line is incident with at least 
four points. If additional incidences do not occur, then we get a $((3n)_4)$ configuration. We call 
these configurations \emph{Kárteszi configurations} and denote them by $\mathrm{K}(n;\ell ,m)$.


\section{When $\mathrm{K}(n;\ell,m)$ contains extra incidences} \label{sect:exceptional}


Theorem \ref{thm:chord-belt} is true for any 
$\ell ,m \in \{1,2,\dots , \lfloor\frac{n-1}{2}\rfloor \}$, but some vertices of $\mathcal{A}_{\ell }$ or 
$\mathcal{A}_{m}$ could 
incident with other diagonals of  $\mathcal{A}_1 $. In these cases, the result is not a $((3n)_4)$ configuration, 
because some lines are incident with more than $4$ points.

\begin{figure}[!h]
\begin{center}
\includegraphics[width=0.66\textwidth]{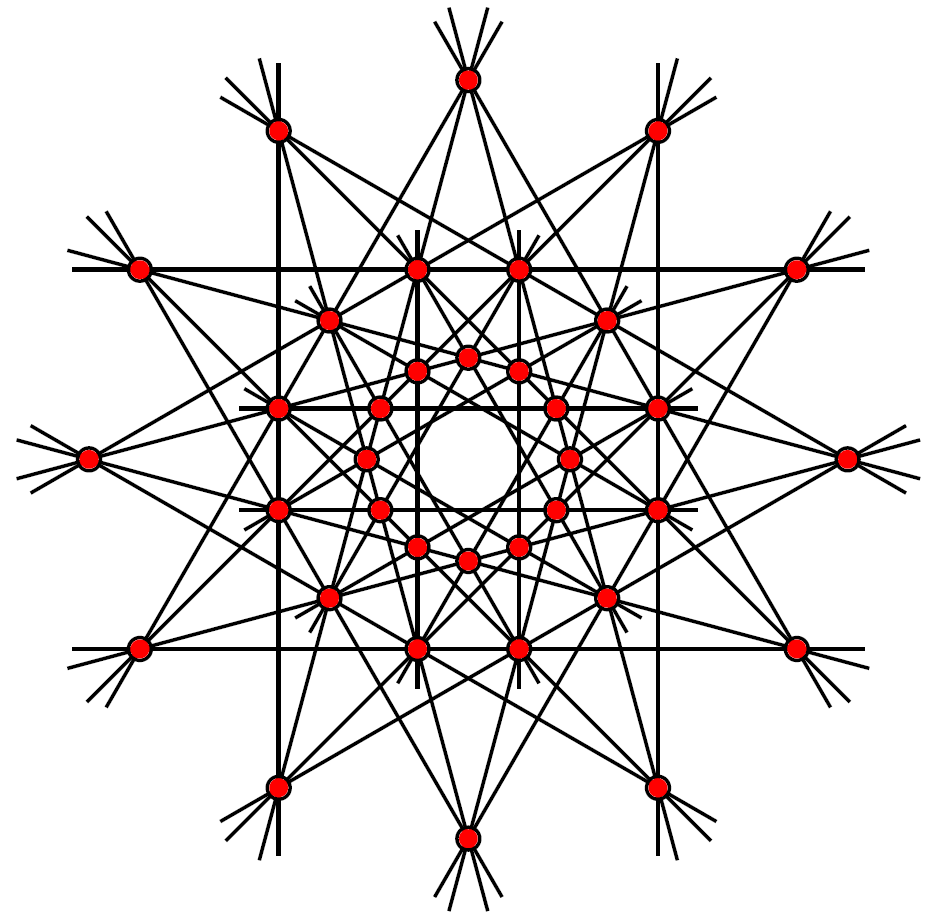}
\caption{The smallest combinatorial Kárteszi configuration $\mathrm{K}(12;4,5)$ has a geometric representation 
with extra incidences.}
 \label{fig:K12;45}
\end{center}
\end{figure}

Poonen and Rubinstein~\cite{PR} presented a complete list when three diagonals of a regular $n$-gon meet. 
In the following, we briefly summarise their results using their notation.

Let $A,B,C,D,E,F$ be six distinct vertices in the order of a regular $n$-gon inscribed into a unit circle 
subdividing the circumference into arcs of length 
$2\pi X$, $2\pi V$, $2\pi Y$, $2\pi W$, $2\pi Z$, $2\pi U$, respectively. 
Then the three diagonals $AD,BE,CF$ meet at a point if and only if
\begin{equation}
\label{3meet}
\sin(\pi U)\sin(\pi V)\sin(\pi W)=\sin(\pi X)\sin(\pi Y)\sin(\pi Z).
\end{equation}

\begin{figure} [!h]
    \centering
    \includegraphics[width=0.66\linewidth]{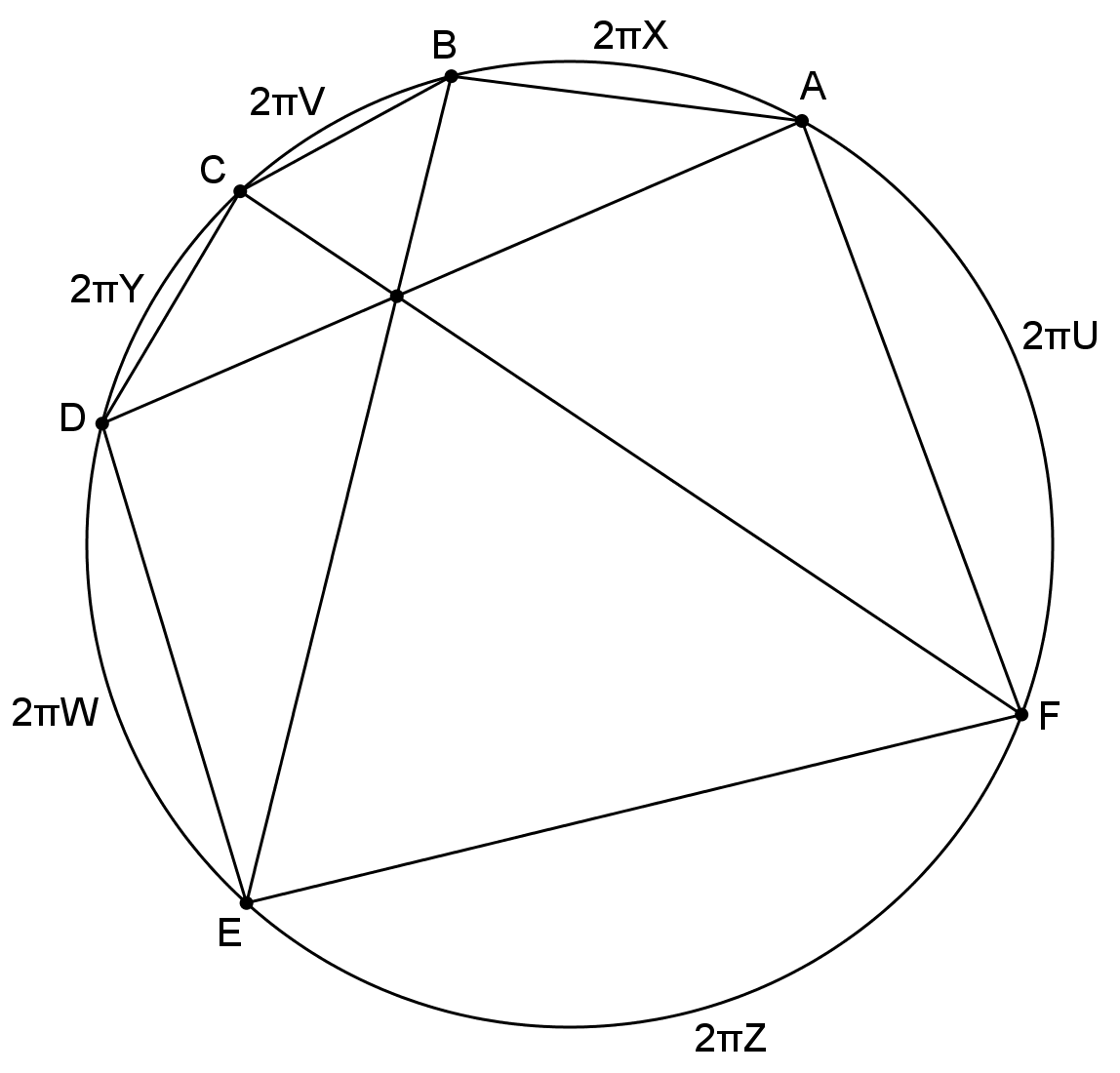}
    \caption{Three meeting diagonals in Theorem of Poonen and Rubinstein.}
    \label{fig:pr}
\end{figure}

They proved the following  \cite[Theorem 4]{PR}.
\smallskip
\begin{thm}[Poonen and Rubinstein, 1998] 
\label{pr-list}
The positive rational solutions of the trigonometric diophantine equation {\rm(\ref{3meet})}
in the case $U+V+W+X+Y+Z=1$ and each of them is a multiple of $\frac{1}{n}$ 
can be divided into three classes.
\begin{enumerate}
\item[$(i)$]
The trivial cases, when $U+V+W=1/2$ and $X,Y,Z$ are a permutation of 
$U,V,W.$
\item[$(ii)$]
Four infinite families, listed in the following table:

\begin{tabular}{||c|c||c||}
\hline
 $\{ U,V,W\} $ &  $\{X,Y,Z\} $ & Range \\
\hline
\hline
$\{ 1/6, t,1/3-2t\} $ & $\{ 1/3+t, t, 1/6-t \} $ & $0<t<1/6$ \\
\hline
$\{  1/6, 1/2-3t, t  \} $  & $\{  1/6-t , 2t, 1/6+t  \} $  & $0<t<1/6 $\\
\hline
$\{ 1/6, 1/6-2t, 2t   \} $ & $\{ 1/6-2t , t , 1/2+t,   \} $ & $0<t<1/12$ \\
\hline 
$\{ 1/3-4t, t, 1/3+t  \} $ & $\{ 1/6-2t, 3t,1/6+t  \} $& $0<t<1/12$ \\
\hline
\end{tabular}

\item[$(iii)$]
$65$ sporadic cases, their list can be found in {\rm \cite[Table 4]{PR}}.
\end{enumerate}
\end{thm}

Using this result, we can specify the cases where 
$\mathrm{K}(n;\ell,m)$ contains extra incidences. 
First, we prove a lemma from elementary geometry.

\begin{lem}
\label{3-bol4}
Suppose that an $r$-th diagonal $d$ of $\mathcal{A}_1 $ is incident with a vertex $B$ of $\mathcal{A}_k$. 
If $B$ is different from the midpoint of $d$, then $d$ is incident with another vertex of $\mathcal{A}_k $, 
hence $d$ has a common line with a diagonal of $\mathcal{A}_k $.

Let $B_i=A_iA_{i+k}\cap A_{i+1}A_{i+1+k}$ denote the vertices of
$\mathcal{A}_k $ (the subscripts are calculated modulo $n$). 
If the diagonal $A_0A_{r}$ is incident with $B_i$, $r<\frac{n}{2},$ $\frac{r}{2}<i$ and 
$2i-r+k+1<\frac{n}{2}$, then $B_{r-i-k-1}$ is also incident with $A_0A_{r}$. 
The $r$-th diagonals of $\mathcal{A}_1 $
have a common line with the $(2i-r+k+1)$-th diagonals of $\mathcal{A}_k $. 
\end{lem}

\begin{proof}
Let $t$ be the perpendicular bisector of $d$. Then $t$ is an axis of symmetry of $\mathcal{A}_1 $, 
hence it is also an axis of symmetry of $\mathcal{A}_k $. Thus, if $\varphi$ denotes the reflection 
in $t$, then $\varphi (B)\neq B$ is another vertex of $\mathcal{A}_k $ and $d$ has a common line 
with the diagonal  $B\varphi (B)$ of $\mathcal{A}_k $.

If $d$ is the diagonal $A_0A_{r}$, then we have $\varphi (A_j)=A_{r-j}$ 
for $j=0,1,\dots ,n-1.$ So 
$$\varphi (B_i)=\varphi (A_{i}A_{i+k}\cap A_{i+1}A_{i+1+k}) =A_{r-i}A_{r-i-k}
\cap A_{r-i-1}A_{r-i-k-1}.$$ 
Since $i+1+k< n-(r-i-k-1)$ we get $A_{r-i-1}A_{r-i-k-1}=B_{r-k-i-1}$. 
Thus, $d$ is incident with $B_{r-k-i-1}$, so it has a common  line with the diagonal 
$B_iB_{r-k-i-1}$ of $\mathcal{A}_k $. This diagonal is a $(2i-r+k+1)$-th diagonal, 
because $2i-r+k+1<n/2.$ By rotational symmetry, this implies that any
$r$-th diagonal of $\mathcal{A}_1 $
has a common line with a suitable $(2i-r+k+1)$-th diagonal of $\mathcal{A}_k $. 
\end{proof}

\begin{thm}
\label{thm:extra-inc}
$\mathrm{K}(n;\ell,m)$ is a $((3n)_4)$ configuration except in the following cases:

\begin{itemize}
\item
$\mathrm{K}(6k; 2k,3k-1)$, $\mathrm{K}(6k;3k-2,3k-1)$ and $\mathrm{K}(6k;2k,3k-2)$  for $k\geq 2.$ 
\item
$\mathrm{K}(12k+6; 3k+1,4k+2)$, $\mathrm{K}(12k+6;3k+2,4k+2)$ and $\mathrm{K}(12k+6;3k+1,3k+2)$ for $k\geq 1.$ 
\item
$\mathrm{K}(30;4,6)$, $\mathrm{K}(30;4,7)$ and $\mathrm{K}(30;6,7)$;
\item
$\mathrm{K}(30;6,10)$, $\mathrm{K}(30;6,11)$ and $\mathrm{K}(30;10,11)$;
\item
$\mathrm{K}(30;8,12)$, $\mathrm{K}(30;8,13)$ and $\mathrm{K}(30;12,13)$;
\item
$\mathrm{K}(42;6,12)$, $\mathrm{K}(42;6,13)$ and $\mathrm{K}(42;12,13)$.
\end{itemize}
\end{thm}

\begin{proof}
If an extra incidence occurs in $\mathrm{K}(n;\ell,m)$, then two of the three diagonals that have a common point 
are consecutive of the same length. Again, using the notation of Figure 1 and the results of Poonen and Rubinstein, 
this implies that the smallest element in both sets 
$\{ X,Y,Z\} $ and $\{ U,V,W\} $ must be $\frac{1}{n}$. 
We may assume without loss of generality that 
$X=W=\frac{1}{n}$. 

Consider the three possibilities of Theorem \ref{pr-list}.
In Case $(i)$, either $Y=U$ and $Z=V,$ or $Y=V$ and $Z=U.$ In the former case $AD,$ in the latter case $CF$ 
is a diameter of the circle, so the trivial solutions do not appear in the Kárteszi configurations.
\medskip

Case $(ii)$. In the first family $t=\frac{1}{n}$, otherwise $\frac{1}{n}=1/3-2t=1/6-t$
would imply $t=1/6.$ Thus there exists an integer $k$ such that $1/6= \frac{k}{n},$ so
$$n=6k, \quad \left\{ U,V\right\} =\left\{ \frac{k}{6k}, \frac{2k-2}{6k}\right\} , 
\quad \left\{ Y,Z\right\} =\left\{ \frac{k-1}{6k}, \frac{2k+1}{6k}\right\} .$$
In the second family the smallest element in $\{ X,Y,Z\} $ is $1/6-t.$ Again,
$\frac{1}{n}=1/3-2t=1/6-t$ implies $t=1/6,$ hence $1/6-t=t,$ so $t=1/12,$ this
results in a trivial solution $\{1/12, 1/6, 1/4\} .$ In the third family the smallest element in 
$\{ X,Y,Z\} $ is either $t$ or $1/6-2t,$ hence the smallest element in $\{ U,V,W\} $  
must be $1/6-2t.$ Thus $1/6-2t=\frac{1}{n}$ and there exists an integer $k$ such that 
$t= \frac{k}{n}.$ So 
$$
n=6(2k+1), \, \left\{ U,V\right\} =
\left\{ \frac{2k+1}{6(2k+1)}, \frac{2k}{6(2k+1)}\right\} , \, \left\{Y,Z\right\} =
\left\{ \frac{k}{6(2k+1)}, \frac{7k+3}{6(2k+1)}\right\}.
$$
In the fourth family the smallest element in $\{ X,Y,Z\} $ is $1/6-2t.$ 
If $\frac{1}{n}=1/6-2t=t,$ then $t=1/18$ and this solution appears in the first family, too. The other case, 
$\frac{1}{n}=1/6-2t=1/3-4t$, 
would imply $t=1/12,$ contradicting condition $t<1/12.$ 
\medskip

Case $(iii)$. Among the 65 sporadic examples, only four contain $\frac{1}{n}$ in both sets.

Now consider the two infinite families and the
$4$ sporadic cases. Because of symmetry, we may assume that 
$$U+X+V=U+V+\frac{1}{n}=U+V+W<\frac{1}{2} \, \text{  and  } \, V\leq U.$$
Then $CF$ is an $r$-th diagonal of the regular $n$-gon with 
$r=n(U+V)+1<\frac{n}{2}$, and  
$AD$ and $BE$ are $\ell$-th diagonals with $\ell =n(V+Y)+1$ if 
$X+V+Y<\frac{1}{2}$ or $\ell =n(1-(V+Y))-1$ if $X+V+Y>\frac{1}{2}$.
Since we have two choices for $Y$, there are two possibilities for $\ell .$ 
If $Y_1<Z_1,$ then $X+V+Y_1<\frac{1}{2}$, so $\ell _1=n(V+Y_1)+1.$
In the other case, $Y_2=Z_1>Y_1=Z_2$ and $X+V+Y_2>\frac{1}{2}$. So
$\ell _2=n(1-(V+Y_2))-1=n(1-(V+Z_1))-1=n(U+Y_1)+1.$
The exceptional triples of diagonals $(r,\ell _1,\ell _2)$ are listed in the following table:

\bigskip
\noindent
\renewcommand{\arraystretch}{1.5}
\begin{tabular}{||c|ccc|c|cc|cc||}
\hline
$n$ & $W$ & $V$ & $U$ &  $r$ & $X$ & $\{ Y,Z\}$ & $\ell _1$ &  $\ell _2$ \\
\hline
\hline
$6k$ & $\frac{1}{6k}$ & $\frac{k}{6k}$ & $\frac{2k-2}{6k}$ & $3k-1$ & 
$\frac{1}{6k}$ & $\{ \frac{k-1}{6k}, \frac{2k+1}{6k}\} $ &  $2k$ &  $3k-2$ \\[4pt]
\hline
$12k+6$ & $\frac{1}{12k+6} $ & $\frac{2k}{12k+6} $ & $\frac{2k+1}{12k+6} $ & 
$4k+2$ & $\frac{1}{12k+6} $ & $\left\{ \frac{k}{12k+6}, \frac{7k+3}{12k+6}\right\} $ & 
$3k+1$ & $3k+2$ \\[4pt]
\hline
30 & $\frac{1}{30}$ & $\frac{2}{30}$ & $\frac{4}{30}$ & 7 & $\frac{1}{30}$ & 
$\{ \frac{1}{30} , \frac{21}{30} \}$ & 4 & $6$ \\
30 & $\frac{1}{30}$ & $\frac{3}{30}$  & $\frac{7}{30}$ & 11 & $\frac{1}{30}$ & 
$\{ \frac{2}{30} , \frac{16}{30} \}$ & 6 & $10$ \\
30 & $\frac{1}{30}$  & $\frac{4}{30}$ & $\frac{8}{30}$ & 13& $\frac{1}{30}$ & 
$\{ \frac{3}{30} , \frac{13}{30} \}$ & 8 & 12 \\[4pt]
\hline
42 & $\frac{1}{42}$ & $\frac{3}{42}$  & $\frac{9}{42}$ & 13 & $\frac{1}{42}$ & 
$\{ \frac{2}{42}, \frac{26}{42} \}$ & 6 & 12 \\[4pt]
\hline
\end{tabular}

\bigskip
If an $r$-th diagonal contains a vertex of $\mathcal{A}_{\ell _1}$, then $r,\ell _1$
and $i=nU$  satisfy the conditions of Lemma \ref{3-bol4}. By this lemma, the $r$-th  diagonal has a common line 
with a $(2i-r+\ell _1+1)$-th diagonal of $\mathcal{A}_{\ell _1}$. Since
$$
2i-r+\ell _1+1=n(2U-( U+V+W) +(X+V+Y_1))+1=n(U+Y_1)+1=\ell_2,
$$
Theorem \ref{thm:chord-belt} implies that this line also contains an $\ell _1$-th diagonal 
of $\mathrm{A}_{\ell _2}$. Hence, six points of the $3n$ vertices of $\mathrm{K}(n;\ell _1,\ell _2)$ are also 
collinear. We get the smallest example in the first family when $k=2$, in this case $2k=3k-2$ so there is only 
one configuration, $\mathrm{K}(12;4,5)$. 
Each other triple $(r,\ell _1,\ell _2)$ corresponds to three K configurations
if $\ell_1\neq \ell _2,$ so in these cases
$\mathrm{K}(n;\ell _1,\ell _2)$, $\mathrm{K}(n;\ell _1,r)$ and $\mathrm{K}(n;\ell _2,r)$ are not $((3n)_4)$ 
configurations. This proves the theorem.
\end{proof}


\section{Conclusion}

In this paper, we present the evidence that the Grünbaum--Rigby configuration has been constructed by 
Ferenc Kárteszi already in 1964. We also point out that a theorem by Kárteszi gives rise to a special 
3-parameter family of trivial 3-celestial 4-configurations that we named \emph{Kárteszi configurations}. 
Using geometric arguments, we give an independent characterization of the validity of these parameters. 

It would be interesting to investigate to what extent the Kárteszi Theorem can be adapted to cover more 
general celestial or even other polycyclic 4-configurations.


\section*{Acknowledgements}
Toma\v{z} Pisanski is supported in part by the Slovenian Research Agency (research program P1-0294 and research projects J1-4351, J5-4596 and BI-HR/23-24-012).
Gy\"orgy Kiss is supported in part by the Slovenian Research Agency, research program N1-0429.
\bigskip



\begin{thebibliography}{99} 


\bibitem{BB}
A.\ Berardinelli and L.\ W.\ Berman,
Systematic celestial 4-configurations,
\emph{Ars Math.\ Contemp.} {\bf 7} (2014), 361--377. 

\bibitem{Be2001}
L.\ W.\ Berman, A characterization of astral $(n_4)$ configurations, 
\emph{Discrete Comput. Geom.} {\bf 26}
(2001), 603--612.

\bibitem{BGP2024}
L.\ W.\ Berman, G.\ G\'evay and T.\ Pisanski,
On a new $(21_4)$ polycyclic configuration,
\emph{Electronic J.\ Combin.} {\bf 31} (4) (2024), \#P4.54.

\bibitem{BoPi2003}
M.\ Boben and T.\ Pisanski,
Polycyclic configurations,
\emph{European J.\ Combin.} {\bf 24} (4) (2003) 431--457.

\bibitem{Cox1950}
H.\ S.\ M.\ Coxeter, 
Self-dual configurations and regular graphs,
\emph{Bull.\ Amer.\ Math.\ Soc.} {\bf 56} (1950), 413--455.


\bibitem{Gru}
B.\ Gr\"unbaum, 
\emph{Configurations of Points and Lines}, 
Graduate Texts in Mathematics, Vol.\ 103,
American Mathematical Society, Providence, Rhode Island, 2009.

\bibitem{GR}
B.\ Gr\"unbaum and J.\ F.\ Rigby, 
The real configuration $(21_4)$, 
\emph{J.\ London Math.\ Soc.} {\bf 41} (1990), 336--346.

\bibitem{HCV1932}
D.\ Hilbert and Cohn-Vossen,
\emph{Anschauliche Geometrie},
Berlin, Springer, 1932.

\bibitem{Kar65}
F.\ K\'arteszi, 
Intorno a punti allineati di certi reticoli circolari, 
\emph{Rend.\ Sem.\ Mat.\ Messina} {\bf 9} (1964/65), 1--12.

\bibitem{Kar1972}
F.\ K\'arteszi, 
\emph{Introduction to Finite Geometries},
Akadémiai Kiadó, 1972 (in Hungarian). 
Reprinted in English by Elsevier/North Holland, 1976, in Italian by Feltrinelli, 1978, and in Russian
by Nauka, 1980.

\bibitem{Kar1986}
F.\ K\'arteszi,
On a surprising analogy,
\emph{Ann.\ Univ. Sci.\ Budapest Eötvös Sect.\ Math.} {\bf 29} (1986), 257--259 (in Italian).

\bibitem{KSz2008}
G.\ Korchmáros and T.\ Szőnyi,
Preface [Special issue dedicated to the centenary of the birth of Ferenc Kárteszi].
\emph{Contrib.\ Discrete Math.} {\bf 3} (1) (2008), 1--2.

\bibitem{KSz2008a}
G.\ Korchmáros and T.\ Szőnyi,
Ferenc K\'arteszi (1907--1989): A short biography,
\emph{Contrib.\ Discrete Math.} {\bf 3} (1) (2008), 3--5.

\bibitem{Levi}
F.\ W.\ Levi, 
\emph{Geometrische Konfigurationen}, Hirzel, Leipzig, 1929.

\bibitem{PS}
T.\ Pisanski and B.\ Servatius,
\emph{Configurations from a Graphical Viewpoint}, 
Birk\-h\"auser Advanced Texts, 
Birkh\"auser, New York, 2013.

\bibitem{PR}
B.\ Poonen and M\ Rubinstein,
The number of intersection points made by the diagonals of a regular polygon,
\emph{SIAM J.\ Discrete Math.} {\bf 11} (1) (1998), 135--156. 

\end{thebibliography}


\end{document}